\documentclass[a4paper,10pt]{article}
\usepackage{amsfonts}
\usepackage[T1]{fontenc}
\usepackage{microtype}
\usepackage{makeidx}
\usepackage{setspace}
\usepackage{graphicx}
\usepackage{epstopdf}
\usepackage[all]{xy}
\usepackage{amsmath}
\usepackage{amssymb}
\usepackage{booktabs}
\usepackage{braket}
\usepackage{array}
\usepackage{tabularx}
\usepackage{multirow}
\usepackage{indentfirst}
\usepackage[nospace,noadjust]{cite}
\usepackage{stmaryrd}
\usepackage{faktor}
\usepackage{titling}
\usepackage{tikz}
\usepackage{dsfont}

\usetikzlibrary{matrix,arrows,decorations.pathmorphing}

\usepackage{amsthm}
\theoremstyle{plain}
\newtheorem{thm}{Theorem}[section]
\newtheorem{dfn}[thm]{Definition}
\newtheorem{lem}[thm]{Lemma}
\newtheorem{prop}[thm]{Proposition}
\newtheorem{cor}[thm]{Corollary}
\newtheorem{ex}[thm]{Example}

\theoremstyle{remark}
\newtheorem{oss}[thm]{Remark}

\DeclareMathOperator{\spn}{span} 
\DeclareMathOperator{\I}{Id}
\DeclareMathOperator{\End}{End}
\DeclareMathOperator{\Aut}{Aut}

\DeclareMathOperator{\modue}{(\mathrm{mod}~2)}

\begin{document}

\begin{center}

{\Large \bf Artin group injection in the Hecke algebra for right-angled groups}

 \vspace{0.5cm}

 Paolo Sentinelli \\

%Departamento de Matem\'aticas \\
%
%Universidad de Chile\\
%
%Las Palmeras 3425, \~Nu\~noa\\
%
%00133 Santiago, Chile \\

{\em paolosentinelli@gmail.com } \\

\end{center}

\begin{abstract}
We prove some injectivity results: that a Coxeter monoid $\mathbb{Z}$-algebra (or $0$-Hecke algebra) injects in the incidence $\mathbb{Z}$-algebra of the corresponding
Bruhat poset, for any Coxeter group; that the Hecke algebra of a right-angled Coxeter group injects in the Coxeter monoid $\mathbb{Z}[q,q^{-1}]$-algebra (and then in the
incidence $\mathbb{Z}[q,q^{-1}]$-algebra of the corresponding Bruhat poset); that a right-angled Artin group injects in the group of invertible elements of the Hecke
algebra of the corresponding Coxeter group (and then in the group of invertible elements of a Coxeter monoid algebra and in the one of an incidence algebra).
\end{abstract}

\section{Introduction}

The HOMFLYPT-polynomials provide an important knot invariant; in the celebrated article of Jones \cite{Jones} they are defined by the Ocneanu's trace on the Hecke algebra of type $A_n$. The trace is
computed on the image of an element of the braid group $B_{n+1}$ under the representation given by the assignment $\sigma_i \mapsto T_{s_i}$, where
$\{\sigma_1,...,\sigma_n\}$
is the set of generators of $B_{n+1}$ and $T_{s_i}$ is a generator of the Hecke algebra of the Coxeter system $(S_{n+1},\{s_1,...,s_n\})$ of type
$A_n$, for all $1\leqslant i \leqslant n$. Here $S_{n+1}$ denotes the symmetric group of order $(n+1)!$. The injectivity of this group morphism is an open problem (except in
small cases). The Burau representation, which is a representation of the Hecke algebra of type $A_n$, solves this
problem for $n<3$ since in these cases the faithfulness (as a representation of the braid group) is known. For $n>3$ the Burau representation is not faithful and the
faithfulness for $n=3$ is unknown (see, e.g., Turaev paper \cite{turaev-braid}). It is worth to mention a paper by  Brunat, Magaard and Marin \cite{Image of the braid groups} devoted to the study of the image of this morphism in the finite field case. Clearly the image is a finite group and then the morphism in not injective.

%The Burau representation doesn't help
%to solve this problem since it is not faithful for $n\geqslant 5$ (see, e.g., \cite{turaev-braid}).

In general, the assignment of the generators of an Artin group to the respective generators of the Hecke algebra of a Coxeter system of same type, furnishes a group morphism
of the Artin group to the group of invertible elements of this Hecke algebra. The injectivity of this morphism seems to be a natural problem. In this article we prove that
such a morphism is injective for the class of right-angled Artin groups (sometimes known as \emph{graph groups}). We refer to \cite{Charney righ-anlged artin} and \cite{Wise} for a wide exposition of problems where such groups
appear; the reader can appreciate their relevance in topology and geometry.
%Although our result is quite easy to prove, there are some interesting injectivity results related to it.

The central argument of the proof of our results lies in the existence of an integral faithful representation of the Coxeter monoid $\mathbb{Z}$-algebra ($0$-Hecke algebra)
of any Coxeter system $(W,S)$, made by idempotent functions which are the projections $P^J: W \rightarrow W^J$ over a set $W^J$ of representative of the quotient of the
Coxeter group $W$ by a parabolic subgroup $W_J$. In particular we use the fact that, viewing these projections as endomorphisms of the free $\mathbb{Z}$-module generated by
$W$, the endomorphisms $P^w:=P^{\{s_1\}}P^{\{s_2\}}\cdots P^{\{s_k\}}$
corresponding to a reduced expression $s_1s_2\cdots s_k$ for $w\in W$ depend only on the elements of $W$ and they are linearly independent.
An important observation is that an endomorphism $P^w$ can be considered as an element of the incidence algebra of the Bruhat poset of $W$.
This results permit to prove, when $R^A$ is a right-angled Artin group generated by $\{s_1,...,s_n\}$, that the assignments $s_i \mapsto q\I-(q+1)P^{s_i}$ and $s_i \mapsto
-\I+(q+1)P^{s_i}$, for all $i\in \{1,2,...,n\}$,
provide the stated injection of $R^A$ in the Hecke algebra $\mathcal{H}(R)$, consequently in the Coxeter monoid algebra over the ring $\mathbb{Z}[q,q^{-1}]$ (see
Theorems~\ref{rep faith} and \ref{rappresentazioni Artin}), and in the incidence algebra of the Bruhat poset of $R$ (Corollary \ref{corollario iniezioni}).

Our results provide also a class of finite-dimensional representations of any right-angled Artin group; in fact these groups act, via the $0$-Hecke algebras in which are embedded, on any lower Bruhat interval of the corresponding Coxeter groups (see the end of Section \ref{4}).

In the finite case the representation theory of the $0$-Hecke algebras was initiated and extensively studied by Norton in \cite{Norton}. A realization of these algebras by projections
over the parabolic quotients was already pointed out and investigated (see, e.g, \cite{HirSchilThie} and \cite{biHecke}). Since we are interested in the infinite case, we
have developed the theory for arbitrary Coxeter systems.

The interest in the Coxeter monoids and, mostly, in the Coxeter monoid algebra of finite monoids is evident looking at the wide literature.  Besides the cited one of P. N.
Norton, general results can be found in \cite{Fayers2005} , \cite{He-subalgebra}, \cite{He-center}, \cite{KenneyCoxetermonoid}, \cite{Tsaranov monoid}. In type $A$ we can
quote, among others, \cite{Carter} and \cite{denton2}. Various actions of the $0$-Hecke algebra of type $A$ are constructed in  \cite{Fomin}, \cite{hivert novelli thibon},
\cite{Huang1}, \cite{Huang2},   \cite{KrobThibon}, \cite{McNamara}, \cite{tewari}, with results related to quasisymmetric functions and noncommutative symmetric functions.
More general results in the setting of representation theory of monoid algebras can be found, e.g., in \cite{denton1} and \cite{quivers monoids}.

The content of the paper is arranged in the following way. Section \ref{1} is devoted to establish notation and to recall known definitions and results
used in the ensuing sections. In Section \ref{2} we show some properties of the projections $P^J: W \rightarrow W^J$; in particular we prove that
two projections commute when acting on a finite Coxeter group if and only if they commute on the maximum of the group.
This result can be useful for computational purpose; in \cite{sentinelli-right} it is used to implement the non-commuting graph of the projections
$P^J$ in type $A_n$; this graph is conjectured to be $n$-universal \cite[Conjecture~4.5]{sentinelli-right}.

The algebra $M$ generated by the set of projections $\{P^J:J\subseteq S\}$ and an integral representation of the Coxeter monoid algebra, obtained with the assignment $s \mapsto
P^{\{s\}}$ for all $s\in S$, are the subjects of Section \ref{3}. We see that for finite Coxeter groups the algebra $M$ is isomorphic to the Coxeter monoid algebra;
in the infinite case the second one is isomorphic to a proper subalgebra of $M$.
In fact, the Coxeter monoid algebra is realized as the algebra generated by the set of idempotents $\{P^{\{s\}}:s\in S\}$ (Theorem \ref{isomorfismi di algebre}); moreover this algebra injects in the incidence algebra of the poset $(W,\leqslant)$, where $\leqslant$
is the Bruhat order (Corollary \ref{cor inject 1}).
Section \ref{4} presents the main results of this article, i.e. the injection of a right-angled Artin group $R^A$ in the Hecke algebra of the Coxeter group of same type
(Theorem \ref{rappresentazioni Artin}), the injection of this Hecke algebra in the Coxeter monoid $\mathbb{Z}[q,q^{-1}]$-algebra (Proposition \ref{rep faith}), which injects
in the incidence $\mathbb{Z}[q,q^{-1}]$-algebra of the Bruhat poset $(R,\leqslant)$ (Corollary \ref{cor inject 2}). As a consequence we obtain, for any element $v$ of its
corresponding right-angled Coxeter group $R$, a representation  of the group $R^A$ over the field $\mathbb{Q}$, of dimension $|\{z\in R:z\leqslant v\}|<\infty$.

\section{Notations and preliminaries} \label{1}
In this section we establish some notation and we collect some basic
results from the theory of Coxeter systems, Coxeter monoids  and Hecke algebras
which will be useful in the sequel. The reader can consult
\cite{BB} and \cite{Hum} for further details. For the isomorphism problem of Coxeter systems we refer to \cite{BahlsTheIsomorphism}. We follow \cite[Chapter~3]{StaEC1} for
notation and terminology concerning posets and \cite{Diestelgraph} for graphs. For the general theory of
ordered monoids and representations of finite monoids the reader can consult \cite{Blyth-ordered structures} and
\cite{Steinberg-Monoid} respectively.

We let $\mathbb{Z}$ be the ring of integers and $\mathbb{Q}$ the field of rational numbers. With $\mathbb{N}$ we denote the set of non-negative
integers. For any $n\in \mathbb{N}$ we let $[n]:=\Set{1,2,...,n}$; in
particular $[0]=\varnothing$. With $\biguplus$ we denote the disjoint union, with $|X|$ the cardinality of a set $X$ and with $\mathcal{P}(X)$ its power set. Given any
category, $\End(O)$ and $\Aut(O)$
denote the set of endomorphisms and automorphisms of an object $O$ respectively. The category of posets is the one
whose objects are posets and whose morphisms are order preserving functions. Given a poset $(P,\leqslant)$, any pair $(x,y)\in P \times P$ satisfying
$x\leqslant y$ defines an \emph{interval}  $[x,y]:=\{z\in P:x\leqslant z \leqslant y\}$. The set of intervals of $(P,\leqslant)$ is denoted with $\mathrm{Int}(P)$;
the poset is called \emph{locally finite} if $|[x,y]|<\infty$ for all $[x,y]\in \mathrm{Int}(P)$. The \emph{incidence algebra}  $I(P;Z)$, over a ring $Z$,  of a
locally finite poset $(P,\leqslant)$ is the $Z$-algebra of {functions}\footnote{We write $f(x,y)$ for $f([x,y])$.} $f: \mathrm{Int}(P) \rightarrow Z$, whose product is
defined by
$$(fg)(x,y):=\sum \limits_{z\in [x,y]}f(x,z)g(z,y),$$ for all $f,g \in I(P;Z)$, $[x,y]\in \mathrm{Int}(P)$. When $P$ is finite, the incidence algebra is isomorphic to a subalgebra of the algebra of upper
triangular matrices with coefficients in the ring $Z$, where an isomorphism is given once any linear extension of the poset is fixed (see \cite{Incidence algebras} for
further deepening on incidence algebras).

Let $(W,S)$ be a Coxeter system. This is a presentation of the group $W$ given by a set $S$ of involutive generators and relations encoded by a \emph{Coxeter matrix}
$m:S\times S \rightarrow \{1,2,...,\infty\}$ or, equivalently, by a \emph{Coxeter graph} (see \cite[Chapter~1]{BB}). A Coxeter matrix over $S$ is a symmetric matrix which
satisfies the following conditions:
\begin{enumerate}
  \item $m(s,t)=1$ if and only if $s=t$;
  \item $m(s,t)\in \{2,3,...,\infty\}$, if $s\neq t$,
\end{enumerate} for all $s,t\in S$. The Coxeter graph associated to a Coxeter system $(W,S)$ with Coxeter matrix $m$, is a labeled graph whose vertices are the elements of
$S$, whose edges are given by the sets $\{s,t\}$ such that $m(s,t) > 2$, displaying the label $m(s,t)$ whenever $m(s,t)\geqslant 4$, for all $s,t\in S$. The presentation
$(W,S)$ of the group $W$ is then the following:
$$\left\{
  \begin{array}{ll}
    \mathrm{generators}: & \hbox{$S$;} \\
    \mathrm{relations}: & \hbox{$(st)^{m(s,t)}=e$,}
  \end{array}
\right.$$ for all $s,t\in S$, where $e$ denotes the identity in $W$. The Coxeter matrix $m$ attains the value $\infty$ at $(s,t)$ to indicate that there is no relation
between the generators $s$ and $t$.

The elements of the group $W$ with the given Coxeter presentation
can be viewed as words in the alphabet $S$; the class of words expressing an element of $W$ contains words of minimal length; the \emph{length function} $\ell: W \rightarrow
\mathbb{N}$ assigns to an element $w\in W$ such a minimal length. The identity $e$ is represented by the empty word and then $\ell(e)=0$. A word of minimal length, expressing
an element $w\in W$, is called a \emph{reduced word} or \emph{reduced expression} for $w$.
If $J\subseteq S$, we let
\begin{gather*} W^J:=\Set{w\in W:\ell(ws)>\ell(w)~\forall~s\in J},
\\ {^JW}:=\Set{w\in W:\ell(sw)>\ell(w)~\forall~s\in J},
\\ D_L(w):=\Set{s\in S:\ell(sw)<\ell(w)},
\\ D_R(w):=\Set{s\in S:\ell(ws)<\ell(w)}.
\end{gather*}
Let $w\in W^J$. It is useful to recall that exactly one of the following three possibilities occurs (see \cite[Lemma~3.1]{deodharII}):
\begin{enumerate}
  \item $s \in D_L(w)$. In this case $sw\in W^J$.
  \item $s \not \in D_L(w)$ and $sw\in W^J$.
  \item $s \not \in D_L(w)$ and $sw \not \in W^J$. In this case $sw=ws'$ for a unique $s'\in J$.
\end{enumerate}
By definition $W^I \cap W^J = W^{I\cup J}$ and ${^I}W \cap {^J}W = {^{I\cup J}}W$. In the literature, the elements of the sets $W^J$ and ${^J}W$
are sometimes called \emph{reduced}-$J$ and $J$-\emph{reduced} respectively.

With $W_J$ we denote the subgroup of $W$ generated by $J\subseteq S$; such a group is usually called a \emph{parabolic subgroup}. In particular $W_S=W$ and $W_\varnothing =
\Set{e}$. We say that the set $J$ is \emph{connected} if the Coxeter graph of $(W_J,J)$ is connected.

When the group $W_J$ is finite, there exists a unique element $w_0(J)$ of maximal length and $D_L(w_0(J))=D_R(w_0(J))=J$ (\cite[Proposition~2.3.1]{BB}). When $J=S$ we write
$w_0$ instead of $w_0(S)$.

Given a Coxeter presentation $(W,S)$, we consider on $W$ the Bruhat order $\leqslant$ (see, e.g., \cite[Chapter~2]{BB} or \cite[Chapter~5]{Hum}).
Such an order can be defined in the following way: let $u,v \in W$ and $s_1s_2 \cdots s_k$ be a reduced word for $v\in W$. Then $u \leqslant v$ if and only if a word
expressing $u$ can be obtained deleting some generators in the reduced word $s_1s_2 \cdots s_k$.

We recall a characterizing property of the
Bruhat order, known as  \emph{lifting property} (see \cite[Proposition~2.2.7 and Exercise~2.14]{BB}):
\begin{prop} \label{sollevamento} Let $v,w\in W$ such that $v<w$ and $s\in D_R(w)\setminus D_R(v)$. Then $v\leqslant
ws$ and $vs\leqslant w$.
\end{prop}

%The next result is a particular case of \cite[Lemma~2.2.10]{BB}.
%\begin{lem} \label{lemma xy}
%  Suppose that $x < xs$ and $y < sy$, for $x, y \in W$, $s \in S$.
%Then, $xy < xsy$.
%\end{lem}

For any $J\subseteq S$, each element $w\in W$ factorizes uniquely as
$w=w^Jw_J$, where $w^J\in W^J$ and $w_J\in W_J$ and $\ell(w)=\ell(w^J)+\ell(w_J)$ (\cite[Proposition~2.4.4]{BB}). We consider the idempotent function $P^J:W
\rightarrow W$ defined by
\begin{equation*} P^J(w)=w^J,
\end{equation*} for all $w\in W$. This function is a morphism of
posets (\cite[Proposition~2.5.1]{BB}):
\begin{prop} \label{ordinepreservato} Let $v,w\in W$ be such that $v \leqslant w$; then $v^J\leqslant
w^J$, for all $J\subseteq S$.
\end{prop}
In a similar way one defines the projection $Q^J:W \rightarrow
{^JW}$ by $Q^J(w)={^Jw}$. The analogues of the last two results hold
for $Q^J$. Summarizing, an element $w\in W$ has unique expressions
\begin{equation} \label{fattorizzazioni} w=P^J(w)P_J(w)=Q_J(w)Q^J(w), \end{equation}
where the maps $P_J,Q_J:W \rightarrow W_J$ are defined in the
obvious way. By~\eqref{fattorizzazioni} follows that $P^J(w),P_J(w),Q^J(w),Q_J(w)\leqslant w$ for all $w\in W$, $J\subseteq S$. So by Proposition \ref{ordinepreservato} the
functions $P^J$, $P_J$, $Q^J$ and $Q_J$ are \emph{regressive} order preserving functions for the poset $(W,\leqslant)$ (see \cite[Definition~2.7]{denton1}).

The following result, and its right version, will be useful in the sequel.
\begin{lem} \label{proj fascio 1 e 2}
Let $(W,S)$ be a Coxeter system and $I \subseteq J \subseteq S$. Then
\begin{enumerate}
  \item $P^J \circ P^I=P^J$;
  \item $P_I \circ P_J=P_I$.
\end{enumerate}
\end{lem}
\begin{proof}
Let $w\in W$. We have that $w = w^Iw_I =w^Jw_J= w^J(w_J)^I(w_J)_I$. Since $w^J(w_J)^I \in W^I$ because $s\in I$ implies $(w_J)^I < (w_J)^Is \in W_J$,
we have that $w^I=w^J(w_J)^I$ and $w_I=(w_J)_I$. From the equality $w^I=w^J(w_J)^I$ also follows that $(w^I)^J=w^J$.
\end{proof}

\begin{oss}Analogous properties of the projections $P^J$  are satisfied also by the
\emph{parabolic map} defined and studied in \cite{billey}.
When the group $W_J$ is finite a function $P^{\setminus J}: W \rightarrow W$ can be defined
by $$P^{\setminus J}(w)=P^J(w)w_0(J),$$ for all $w\in W$. It is easy to see that $P^{\setminus J}$ is idempotent and order preserving. The idempotents $P^J$ together with the
idempotents $P^{\setminus J}$
generate the biHecke monoid (see \cite[Section~1]{biHecke}, where these operators are respectively called \emph{bubble sorting operators} and \emph{bubble antisorting
operators}). In Section \ref{3} we use the idempotents $P^J$ to realize Coxeter monoid algebras, although a realization by bubble antisorting
operators is also possible. \end{oss}

Another property of the projections on $W^J$ and ${^J}W$ is that the right projections commute with the left ones (for a proof of this result see
\cite[Lemma~2.6]{sentinelli-isomoprhism}).
\begin{lem} \label{proiezionicommutano} Let $I,J\subseteq S$; then the projections $P^J$ and
$Q^I$ commute, i.e.
$$P^J\circ Q^I = Q^I \circ P^J.$$
\end{lem} %The sets $W^J$ and ${^J}W$  with the induced Bruhat order are graded posets with the length $\ell$ as rank function,
%with minimums $e$ and, when $|W|<\infty$, with maximums $P^J(w_0)$ and $Q^J(w_0)$
%respectively (see \cite[Chapter~2]{BB}).

Given a Coxeter system $(W,S)$ with Coxeter matrix $m:S\times S \rightarrow \{1,2,...,\infty\}$, the corresponding
Coxeter monoid $W^M$ is the monoid with identity $e$ generated by the set $S$, satisfying the following relations:

$$\left\{
    \begin{array}{ll}
      s^2=s; & \\
      (st)^{m(s,t)/2}=(ts)^{m(s,t)/2}, & \hbox{if $m(s,t) \equiv 0\modue$;} \\
      t(st)^{(m(s,t)-1)/2}=s(ts)^{(m(s,t)-1)/2}, & \hbox{if $m(s,t) \equiv 1\modue$,}
    \end{array}
  \right.
$$ for all $s,t\in S$. Note that as sets $W=W^M$. The following definition establishes the notion of ordered monoid.
\begin{dfn}
  A poset $(M,\leqslant)$ is an \emph{ordered monoid} if $M$ is a monoid and
  $x_1\leqslant y_1$, $x_2\leqslant y_2$ implies $x_1x_2\leqslant y_1y_2$,
  for all $x_1,x_2,y_1,y_2 \in M$.
\end{dfn} Although a Coxeter group $W$ with the Bruhat order is not an ordered monoid, the Coxeter monoid is ordered (see~\cite{KenneyCoxetermonoid} for further results on
Coxeter monoids and, in particular, \cite[Lemma~2]{KenneyCoxetermonoid}).
\begin{prop} \label{coxeter monoid is ordered}
  The Coxeter monoid $W^M$ with the Bruhat order is an ordered monoid.
\end{prop}
Let $A:=\mathbb{Z}[q^{-1},q]$ be  the ring of Laurent
polynomials in the indeterminate~$q$. For any Coxeter system $(W,S)$, the \emph{Hecke
algebra} $\mathcal{H}(W,S)$ is the free $A$-module with basis
$\Set{T_w:w\in W}$ and product defined by
\begin{equation*}T_wT_s=\begin{cases}  T_{ws}, & \mbox{if $s\not \in
D_R(w)$,} \\  qT_{ws}+(q-1)T_w, & \mbox{otherwise,}
\end{cases}\end{equation*} for all $w\in W$ and $s\in S$.  For $s\in S$ one can easily see that \begin{equation*} T_s^{-1}=(q^{-1}-1)T_e+q^{-1}T_s
\end{equation*} and then use this to invert all the elements $T_w$, where $w\in W$. On $\mathcal{H}(W,S)$ there is an involution $\iota$,
as defined in \cite{kazhdanlusztig}, such that
\begin{equation*}\label{involuzionealgebra}\iota(q)=q^{-1},~~ \iota(T_w)=T^{-1}_{w^{-1}}, \end{equation*}
for all $w\in W$. Furthermore (see, e.g., \cite{Hum}) this function is a ring automorphism, i.e.
 \begin{equation*} \label{iotamorfanelli}
\iota(T_vT_w)=\iota(T_v)\iota(T_w),
\end{equation*} for all $v,w\in W$.

\begin{dfn}
  Given a Coxeter system $(W,S)$, the \emph{Coxeter monoid algebra} $Z[W^M]$ is the monoid algebra over the ring $Z$ of the Coxeter monoid $W^M$.
\end{dfn}

\begin{oss}
 The $0$-Hecke algebra is the specialization of $\mathcal{H}(W,S)$ at
$q=0$ and it is isomorphic to the Coxeter monoid algebra $\mathbb{Z}[W^M]$, as one can see via the isomorphism defined by $T_s\mapsto -s$.
\end{oss}

We end this section recalling some facts about right-angled Coxeter and Artin groups. For further deepening on these groups and their relevance
in geometry and topology one can consult the books \cite{Davis} and \cite{Wise}, and the paper \cite{Charney righ-anlged artin}.
\begin{dfn}
  Let $(W,S)$ be a Coxeter system with Coxeter matrix $m: S\times S \rightarrow \{1,2,...,\infty\}$.
 The system $(W,S)$ is called \emph{right-angled} if $\{\infty\} \subseteq \{m(s,t):s,t \in S\}\subseteq \{1,2,\infty\}$.
\end{dfn}
Given a Coxeter system $(W,S)$ with Coxeter matrix $m: S\times S \rightarrow \{1,2,...,\infty\}$, the Artin group $W^A$ of type $(W,S)$
is the group given by the following presentation:
$$\left\{
  \begin{array}{ll}
    \mathrm{generators}: & \hbox{$S$;} \\
    \mathrm{relations}: & \hbox{$\left\{
                                                       \begin{array}{ll}
                                                         (st)^{m(s,t)/2}=(ts)^{m(s,t)/2}, & \hbox{if $m(s,t)\equiv 0 \modue$;} \\
                                                         t(st)^{(m(s,t)-1)/2}=s(ts)^{(m(s,t)-1)/2}, & \hbox{if $m(s,t)\equiv 1 \modue$.}
                                                       \end{array}
                                                     \right.
$.}
  \end{array}
\right.$$

If $(W,S)$ is right-angled then the Artin group $W^A$ is called right-angled. We refer to \cite[Section~4]{geometry right angled artin} for the following facts about right-angled
Artin groups. Let $w\in R^A$ be an element of a right-angled
Artin group $R^A$. Then $w=s_1^{e_1} \cdots s_k^{e_k}$, for some $s_1,...,s_k \in S$, $e_1,...,e_k \in \mathbb{Z}$.
An element $s_i^{e_i}$ is called \emph{syllable} of $w$. The following moves can be applied to $w$:

\begin{enumerate}
  \item remove a syllable $s_i^{e_i}$ if $e_i=0$;
  \item if $s_i=s_{i+1}$ then replace $s_i^{e_i}s_{i+1}^{e_{i+1}}$ by $s_i^{e_i+e_{i+1}}$;
  \item if $s_is_{i+1}=s_{i+1}s_i$ then replace $s_i^{e_i}s_{i+1}^{e_{i+1}}$ with $s_{i+1}^{e_{i+1}}s_i^{e_i}$.
\end{enumerate}

We say that a word representing $w \in R^A$ is \emph{reduced} if its number of syllables is minimal.
The following result holds (see \cite[Theorem~4.1]{geometry right angled artin} and references there):

\begin{thm} \label{teorema artin}
  Any two words representing $w\in R^A$ can be connected via a sequence of the moves above. In particular, if two words are reduced,
  they  can be connected via a sequence of moves of third type.
\end{thm}

A Coxeter group $W$ is said to be \emph{rigid} if, given two Coxeter systems $(W,S)$ and $(W,T)$ there exists
an element $\phi \in \Aut(W)$ such that $\phi(s)\in T$ for all $s\in S$. If $W$ is rigid, the Coxeter system $(W,S)$ (and so the Bruhat order and the Hecke algebra) is
uniquely determined
by the group $W$, modulo automorphisms of its Coxeter graph. The following statement asserts the rigidity of a right-angled Coxeter system (see
\cite[Theorem~3.1]{BahlsTheIsomorphism}).
\begin{thm} \label{rigidità}
  Let $(W,S)$ be a right-angled Coxeter system. Then $W$ is rigid.
\end{thm} By Theorem \ref{rigidità}, in the right-angled case one can speak about the Bruhat order of the group $W$ and the Hecke algebra of the group $W$, without any
specification of its Coxeter presentation.

\section{Some properties of the projections $P^J$} \label{2}

Given a Coxeter system $(W,S)$ let $V_W$ be the free $\mathbb{Z}$-module with basis the set $W$.
Any projection $P^I:W\rightarrow W^I$ extends to an idempotent endomorphism $P^I \in \End(V_W)$;  in the sequel we will not distinguish between functions
 from $W$ to $W$ and endomorphisms of $V_W$.

\begin{dfn}
 Given a Coxeter system $(W,S)$, we define the $\mathbb{Z}$-algebra $M(W,S)$ as the subalgebra of $\End(V_W)$ generated by the set of idempotents $\{P^I:I \subseteq S\}$.
\end{dfn}

By the regressivity of the projections $P^J$,
in the finite case any linear extension of the Bruhat order on $W$ furnishes a representation of the algebra $M(W,S)$, made of triangular matrices with spectrum lying in $\{0,1\}$ and identity
given by $P^{\varnothing}$.
The algebra $M(W,S)$ is a subalgebra of the monoid algebra of regressive order preserving functions (see \cite[Section~2.5]{denton1}).
%Considering the commutator of this algebra as a Lie bracket, when $W$ is finite the corresponding Lie algebra is solvable, being isomorphic to a Lie subalgebra of the Lie

For $I,J \subseteq S$ we use the notation $[I,J]=0$ if $m(s,t)\in \{1,2\}$ for all $s\in I$, $t\in J$, where $m$ is the Coxeter matrix of $(W,S)$.
Otherwise we write $[I,J]\neq 0$. We use the same notation $[\cdot,\cdot]$ for the Lie bracket on the algebra $\End(V_W)$.

The next lemma is useful to prove some properties of the endomorphisms $P^J$ and to characterize the projections commuting on a finite group.
\begin{lem} \label{lemmaPPw0}
  Let $(W,S)$ be a Coxeter system and $w \in W$. Then
  $$P^IP^Jw=w^{I\cup J}P^IP^Jw_{I\cup J},$$ for all $I,J\subseteq S$.
\end{lem}
\begin{proof}
Since $w=w^Jw_J=w^{I\cup J}w_{I\cup J}$, we have
\begin{eqnarray*}
  w &=& (w^J)^I(w^J)_Iw_J \\
   &=&  w^{I\cup J}((w_{I\cup J})^J)^I ((w_{I\cup J})^J)_I (w_{I\cup J})_J \\&=&w^{I\cup J}((w_{I\cup J})^J)^I ((w_{I\cup J})^J)_I w_J,
\end{eqnarray*} where we have used Lemma \ref{proj fascio 1 e 2} to obtain the last equality. It is clear that $w^{I\cup J}((w_{I\cup J})^J)^I\in W^I$,
so the result follows.
\end{proof}

In the following proposition we give necessary and sufficient conditions for projections to commute.

\begin{prop} \label{Pcommutanti}
Let $(W,S)$ be a Coxeter system and $I,J \subseteq S$ connected. The following are equivalent:
\begin{enumerate}
  \item $[P^I,P^J] = 0$;
  \item $[I,J]=0$ or $I \cap J \in \{I,J\}$.
\end{enumerate}

\end{prop}
\begin{proof}

If $I \subseteq J$, then $W^J\subseteq W^I$ so $P^IP^J=P^J$; moreover, by Lemma \ref{proj fascio 1 e 2}, we have $P^JP^I=P^J$. Therefore
$[P^I,P^J] = 0$.

Let $[I,J]=0$; then $I=J=\{s\}$ or $I\cap J=\varnothing$, by the connectedness of $I$ and $J$. In the first case the result is obvious.
Let us consider the second case. For $w\in W$ we have $w_{I\cup J} = w_Iw_J=w_Jw_I$ and then, by Lemma \ref{lemmaPPw0},
$P^IP^Jw=P^JP^Iw=w^{I\cup J}$, i.e. $P^IP^J=P^JP^I=P^{I\cup J}$.
%$u\in W$ be a minimal element such that $P^IP^Ju\neq P^JP^Iu$. In particular $u\not \in W^{I\cup J}$. If $u\not \in W^J$ there is a unique factorization $u=u^Ju_J$ with
%$u_J\neq e$; so $P^Ju<u$ and we have
%  \begin{eqnarray*}
%  % \nonumber % Remove numbering (before each equation)
%    P^JP^Iu^Ju_J &=& P^JP^I(u^J)^I(u^J)_Iu_J \\
%    &=& P^JP^I(u^J)^Iu_J(u^J)_I \\
%    &=& P^JP^I(u^J)^Iu_J \\
%    &=& P^JP^IP^Ju,
%  \end{eqnarray*}
%  since $P^I(u^J)^Iu_J=(u^J)^Iu_J$ and $P^J(u^J)^Iu_J=P^J(u^J)^I$. In fact $u_J<u_Js=su_J$  and $(u^J)^I<(u^J)^Is$, for all $s\in I$, and then,  by Lemma \ref{lemma xy},
%  $(u^J)^Iu_Js=(u^J)^Isu_J>(u^J)^Iu_J$, for all $s\in I$. Hence $(u^J)^Iu_J\in W^I$.
%
%By the minimality of $u$ we obtain $P^JP^Iu=P^JP^IP^Ju=P^IP^Ju$, a contradiction.
%  The same argument shows that if $u\not \in W^I$ then $P^IP^Ju= P^JP^Iu$. Therefore $P^IP^J= P^JP^I$.

Now let $[P^I,P^J] = 0$ and $[I,J]\neq 0$. If $I \cap J \not \in \{I,J\}$ let $s \in I \setminus J$ and $t\in J \setminus I$ be such that $[\{s\},J] \neq 0$ and
        $[\{t\},I] \neq 0$. By connectedness there exists
          a path  $s,s_1,s_2,...,s_k,t$ of minimal length in the Coxeter graph of $(W,S)$ connecting $s$ and $t$ such that $s_1,...,s_k\in I\cap J$. Then $P^JP^Its_ks_{k-1}
          \cdots s_1s=P^Jt=e$ and $P^IP^Jts_ks_{k-1} \cdots s_1s=P^Its_ks_{k-1} \cdots s_1s=t$, i.e. $[P^I,P^J]\neq 0$. Hence we conclude that $I \cap J \in \{I,J\}$.

\end{proof}

Let $I\subseteq S$; we say that a projection $P^I \in M(W,S)$ is \emph{connected} if $I$ is connected. In the next proposition we show how any projection factorizes as a
product of connected projections.

\begin{prop} \label{corolPprodotto}
  Let $I = \biguplus\limits_{i=1}^n I_i$ be a partition of $I\subseteq S$ by maximal connected sets.
  Then $P^I=P^{I_1}P^{I_2}\cdots P^{I_n}$.
\end{prop}
\begin{proof}
  By hypothesis $[P^{I_i},P^{I_j}]=0$ for all $i,j\in [n]$; hence $P^I=P^{I_1 \uplus \ldots \uplus I_n}=P^{I_1}P^{I_2}\cdots P^{I_n}$, as in the proof of Proposition \ref{Pcommutanti}.
\end{proof}

By Proposition \ref{corolPprodotto} the algebra $M(W,S)$ is generated by the connected projections.
The following result concerns the general case of projections $P^I$ and $P^J$ when $I$ and $J$ are possibly not connected.

\begin{prop} \label{propNotconn}
  Let $I = \biguplus \limits_{i=1}^m I_i$ and $J = \biguplus \limits_{i=1}^n J_i$ be partitions of $I$ and $J$ by maximal connected sets. Then $[P^I,P^J]=0$ if and only if
  $[P^{I_i},P^{J_j}]=0$ for all $i\in [m]$, $j\in [n]$.
\end{prop}
\begin{proof}
  One implication is obvious. So let $[P^I,P^J]=0$ and $i\in [m]$, $j\in [n]$ be such that $[P^{I_i},P^{J_j}]\neq0$.
  Then, by Proposition \ref{Pcommutanti}, $[I_i,J_j] \neq 0$ and $I_i\cap J_j \not \in \{I_i,J_j\}$. Let $s\in I_i \setminus J_j$ and
  $t\in J_j \setminus I_i$ so that $[\{s\},J_j]\neq 0$ and $[\{t\},I_i]\neq 0$.
  Consider a path  $s,s_1,s_2,...,s_k,t$ of minimal length in the Coxeter graph of $(W,S)$ connecting $s$ and $t$ such that $s_1,...,s_k\in I_i\cap J_j$. Therefore $s \not
  \in J \setminus J_j$, $t\not \in I\setminus I_i$ and $s_1,...,s_k \not \in (I\cup J)\setminus (I_i \cap J_j)$, since
the sets $\{I_1,...,I_m\}$ and $\{J_1,...,J_n\}$ are partitions made by maximal connected subsets of $I$ and $J$ respectively. By Proposition \ref{corolPprodotto} we obtain
  \begin{eqnarray*}
    P^IP^Jts_ks_{k-1}...s_1s &=& P^{I_1}P^{I_2}\cdots P^{I_m}P^{J}ts_ks_{k-1}...s_1s \\
    &=& P^{I_1}P^{I_2}\cdots P^{I_m}ts_ks_{k-1}...s_1s \\
    &=& P^{I_i}ts_ks_{k-1}...s_1s = t,
  \end{eqnarray*}
  and
  $$P^JP^Its_ks_{k-1}...s_1s = P^Jt =e,$$ which is a contradiction.
\end{proof}

Now we characterize the projections commuting on a finite group, testing their commutativity on the maximum of the group.
\begin{prop} \label{prop PPw0}
  Let $(W,S)$ be a Coxeter system such that $|W|<\infty$. Then $$[P^I,P^J]=0 \Leftrightarrow [P^I,P^J]w_0=0,$$ for all $I,J\subseteq S$.
\end{prop}
\begin{proof}
  One implication is trivial. So let $P^IP^Jw_0=P^JP^Iw_0$. %We prove the result by induction on $|S \setminus (I\cup J)|$.
  By Lemma \ref{lemmaPPw0} we deduce that $$P^IP^J w_0(I\cup J) = P^JP^Iw_0(I\cup J) \in W^{I \cup J} \cap W_{I\cup J} = \{e\}.$$

  Let $v\in W$. Then $v_{I\cup J} \leqslant w_0(I \cup J)$. Since the projections are order preserving, we obtain
  $P^JP^I v_{I\cup J} \leqslant P^JP^I w_0(I \cup J)=e$ and $P^IP^J v_{I\cup J} \leqslant P^IP^J w_0(I \cup J)=e$. The result follows again by Lemma \ref{lemmaPPw0}.

\end{proof}

Lemma \ref{lemmaPPw0} and Proposition \ref{prop PPw0} give some characterizations of commuting projections in the finite case. Let us resume these results.

\begin{thm} \label{teoremaComm}
 Let $(W,S)$ be a Coxeter system such that $|W|<\infty$. Then the following statements are equivalent:
  \begin{enumerate}
   % \item $P^IP^Jw_0(I\cup J)=e$;
   % \item $P^IP^Jw_0=P^{I\cup J}w_0$;
    \item  $[P^I,P^J]w_0(I\cup J)=0$;
    \item $[P^I,P^J]w_0=0$;
    \item $[P^I,P^J]=0$.
  \end{enumerate}
\end{thm}

We give an example, referring to \cite[Section~2.4]{BB} for the action of the projection $P^J$ on a permutation written in one line notation.
\begin{ex}
  Let $(S_6,\{s_1,s_2,s_3,s_4,s_5\})$ be the Coxeter system of type $A_5$ realized by the symmetric group $S_6$, generated by the simple transpositions
  $$\{s_1,s_2,s_3,s_4,s_5\}=\{(1,2),(2,3),(3,4),(4,5),(5,6)\}.$$ In one line notation, we have that the element of maximal length in $S_6$ is $w_0=654321$.

  The projection $P^{\{s_1,s_2\}}$ acts on a permutation $\sigma$, written in one line notation, by reordering the set $\{\sigma(1),\sigma(2),\sigma(3)\}$
  and $P^{\{s_2,s_3\}}$ acts by reordering the set $\{\sigma(2),\sigma(3),\sigma(4)\}$.
  Therefore
  $$P^{\{s_2,s_3\}}P^{\{s_1,s_2\}}(654321)=P^{\{s_2,s_3\}}(456321)=435621 $$ and  $$P^{\{s_1,s_2\}}P^{\{s_2,s_3\}}(654321)=P^{\{s_1,s_2\}}(634521)=346521.$$

  The projection $P^{\{s_1,s_5\}}$ acts on a permutation $\sigma$ by reordering the sets $\{\sigma(1),\sigma(2)\}$ and $\{\sigma(5),\sigma(6)\}$
  and $P^{\{s_3,s_5\}}$ acts by reordering the sets $\{\sigma(3),\sigma(4)\}$ and $\{\sigma(5),\sigma(6)\}$.
  Therefore
  $$P^{\{s_3,s_5\}}P^{\{s_1,s_5\}}(654321)=P^{\{s_3,s_5\}}(564312)=563412$$ and  $$P^{\{s_1,s_5\}}P^{\{s_3,s_5\}}(654321)=P^{\{s_1,s_5\}}(653412)=563412.$$
  % Another computation: $P^{\{s_2,s_3\}}P^{\{s_2\}}(4321)=P^{\{s_2,s_3\}}(4231)=4123$ and  $P^{\{s_2\}}P^{\{s_2,s_3\}}(4321)=4123$.
%  Then $[P^{\{s_2\}},P^{\{s_2,s_3\}}]=0$, by Theorem \ref{teoremaComm}.
\end{ex}

We end this section by defining a graph useful for further developments. Given a Coxeter system $(W,S)$
we let $G_2(W,S)$ to be the \emph{non-commutation graph} of the set of idempotents $\{P^J: J \subseteq S\} \setminus \{\I, P^S\}$;
then its set of vertices is $\mathcal{P}(S) \setminus \{\varnothing, S\}$  and $\{I,J\}$ is and edge if and only if $[P^I,P^J]\neq 0$.
We let $G(W,S)$ to be the graph $G_2(W,S)$ with labeled edges. A label $m(I,J)$ is defined by

$$m(I,J):=\left\{
    \begin{array}{ll}
      2n+1, & \hbox{if $(P^IP^J)^n \neq (P^JP^I)^n$} \\
          & \hbox{and $P^J(P^IP^J)^n = P^I(P^JP^I)^n$;} \\
      2n+2, & \hbox{if $P^J(P^IP^J)^n \neq P^I(P^JP^I)^n$} \\
            & \hbox{and $(P^IP^J)^{n+1}=(P^JP^I)^{n+1}$;} \\
      \infty,  & \hbox{otherwise,}
    \end{array}
  \right.$$ for all edges $\{I,J\}$. As for Coxeter graphs, we drop the label $m(I,J)=3$.
So, for example, the graph $G(S_4,[3])$ is the one of Figure \ref{fig:GA3}, where $(S_4,[3])$ is the symmetric group of order $24$ with
its standard Coxeter presentation. By \cite[Theorem~4.1]{sentinelli-right} the graph $G(S_{n+1},[n])$ is $n$-universal for forests;  \cite[Conjecture~4.5]{sentinelli-right}
asserts that it is $n$-universal.

\begin{figure}
\begin{center}
\includegraphics[scale=0.5]{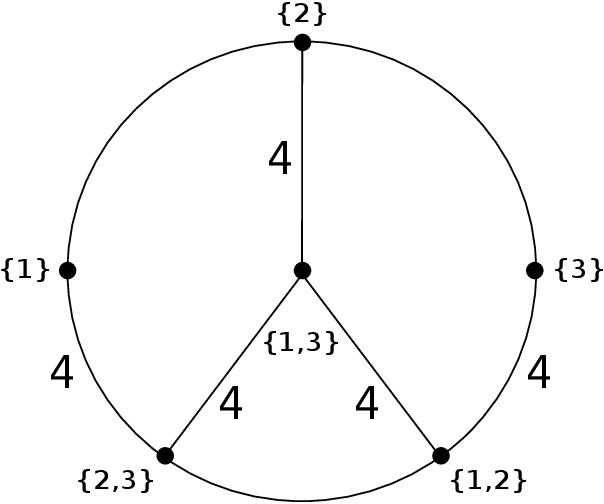}
\caption{$G(S_4,[3])$}
\label{fig:GA3}
\end{center}
\end{figure}

\section{A representation of the Coxeter monoid algebra} \label{3}

Given a Coxeter system $(W,S)$ let us denote with $M_0(W,S)$ the
subalgebra of $M(W,S)$ with identity generated by the set of idempotents $\{P^{\{s\}} \in \End(V_W): s\in S\}$.
In this section we prove that the algebra $M_0(W,S)$ is isomorphic to the Coxeter monoid $\mathbb{Z}$-algebra in the finite and infinite case and that $M_0(W,S)=M(W,S)$ if
$|W|<\infty$. These facts are known in the finite case (see, e.g. \cite{biHecke}). In the infinite case $M_0(W,S)$ is a proper subalgebra of $M(W,S)$.
The representation theory of the Coxeter monoid algebra over a field in the finite case was firstly studied in \cite{Norton}.
Some results of this section, in the finite case, could be deduced from the general theory exposed in the cited paper and in
more recent ones (see, e.g., \cite{denton1} and \cite{biHecke}). To pursue homogeneity and generality we will prove  all the results we need in our setting and notation.

As a consequence of the $\mathbb{Z}$-algebra isomorphism $M_0(W,S) \simeq \mathbb{Z}[W^M]$ we obtain that the Coxeter monoid algebra $\mathbb{Z}[W^M]$ of any type injects in
the incidence algebra $I(W;\mathbb{Z})$ of the Bruhat order. We also define a family of $\mathbb{Z}[W^M]$-modules which gives, in the right-angled case, a family of
finite-dimensional representations of the Artin group $W^A$, as it is shown in the next section.

In order to prove the announced isomorphism, we need the following lemma, which states that the idempotents $\{P^{\{s\}} \in \End(V_W): s\in S\}$ satisfy the relations
encoded in the Coxeter matrix of $(W,S)$.
\begin{lem} \label{lemma st}
  Let $(W,S)$ be a Coxeter system with Coxeter matrix $m$. Then the label of the edge $\{\{s\},\{t\}\}$ in the graph $G(W,S)$ is $m(s,t)$, for all $s,t\in S$.
\end{lem}
\begin{proof}
  Let $m(s,t)$ be even. Then, since $(st)^{m(s,t)/2} = (ts)^{m(s,t)/2}$, we obtain
  \begin{eqnarray*}
    (P^{\{s\}}P^{\{t\}})^k(st)^{m(s,t)/2} &=& (st)^{m(s,t)/2-k}  \\
   &\neq& s(ts)^{m(s,t)/2-k}=(P^{\{t\}}P^{\{s\}})^k(st)^{m(s,t)/2},
  \end{eqnarray*} for all $k<m(s,t)/2$.
  The odd case is analogous. These computations also prove the result if $m(s,t)=\infty$.
  Now let $m(s,t)$ be even and $w\in W$. Then $w=w^{\{s,t\}}w_{\{s,t\}}$ and so $(P^{\{s\}}P^{\{t\}})^{m(s,t)/2}w=w^{\{s,t\}} = (P^{\{t\}}P^{\{s\}})^{m(s,t)/2}w$.
  In the odd case we proceed in the same manner. Therefore in all cases $m(s,t)$ is the label of the edge $\{\{s\},\{t\}\}$, as defined at the end of the previous section.
\end{proof}

 Given $u \in W$ with reduced expression $s_1s_2 \cdots s_k$, let us define
\begin{eqnarray*}
  P^e &:=& \I, \\
  P^u &:=& P^{\{s_1\}}P^{\{s_2\}} \cdots P^{\{s_k\}}.
\end{eqnarray*}

Since a Coxeter group $W$ has the word property, i.e. any two reduced words for $u \in W$ can be connected via a sequence of braid-moves (see, e.g. \cite[Theorem~3.3.1]{BB}),
by Lemma \ref{lemma st} the endomorphism $P^u$ is well defined for all $u\in W$. Notice that such endomorphisms realize endomorphisms of the poset $(W,\leqslant)$, since they
are composition of poset endomorphisms  (see Proposition \ref{ordinepreservato}).

\begin{prop} \label{lemmaPu}
  Let $u,v \in W$. Then $P^vu=e$ if and only if $u\leqslant v$.
\end{prop}
\begin{proof}
  We prove the result by induction on $\ell(v)$. If $\ell(v)=0$ then $P^v=\I$ and so the result is obvious.
Let $\ell(v)>0$ and $s\in D_R(v)$. By Lemma \ref{lemma st} we can write $P^v=P^{vs}P^s$. There are two cases to consider:
\begin{enumerate}
  \item $s\in D_R(u)$: in this case $P^su=us$. Therefore $P^vu=P^{vs}us$ and by the inductive hypothesis $P^{vs}us=e$ if and only if $us\leqslant vs$.
But $us\leqslant vs$ if and only if $u\leqslant v$.
  \item $s \not \in D_R(u)$: in this case $P^su=u$. Hence $P^vu=P^{vs}u$ and by the inductive hypothesis $P^{vs}u=e$
if and only if $u\leqslant vs$. By the lifting property (Proposition \ref{sollevamento}) we have that $u\leqslant vs$ if and only if $u\leqslant v$.
\end{enumerate}
\end{proof}

By the next corollary we can deduce that for finite Coxeter systems the algebra $M(W,S)$ is generated by the idempotents $P^{\{s\}}$.
\begin{cor}
  Let $(W,S)$ be a Coxeter system, $J\subseteq S$ and $|W_J|<\infty$. Then $$P^{w_0(J)} = P^J.$$ In particular,
  $|W|<\infty$ implies $M(W,S)=M_0(W,S)$.
\end{cor}
\begin{proof}
  Let $u=u^Ju_J\in W$ and $w_0(J)=\tilde{u}u_J$, for some $\tilde{u}\in W_J$. Then $P^{w_0(J)}u=P^{\tilde{u}}P^{u_J}u^Ju_J=P^{\tilde{u}}u^J=u^J=P^Ju$.
\end{proof}

Now we are ready to prove that the algebra $M_0(W,S)$ is isomorphic to the monoid algebra over $\mathbb{Z}$ of $W^M$.

\begin{thm} \label{isomorfismi di algebre}
  Let $(W,S)$ be a Coxeter system. Then the function $u\mapsto P^u$ defines an isomorphism of $\mathbb{Z}$-algebras
  $$ \mathbb{Z}[W^M] \simeq M_0(W,S).$$
\end{thm}
\begin{proof}
 Let $a=\sum \limits_{w\in B} a_wP^w \in M_0(W,S)$ with $a_w\in \mathbb{Z}\setminus \{0\}$ for all $w\in B$, where $B$ is any finite subset of $W$. Then there exists a set
 $M(B):=\{v_1,...,v_k\}$ of
maximal elements in $B$. Therefore, if $a=0$, we have that for every $v\in M(B)$,
$av=a_ve+\sum \limits_{w\in B\setminus\{v\}} a_wP^wv =0$, where $P^wv\neq e$ for all $w\in B\setminus\{v\}$ (by Proposition \ref{lemmaPu}); this implies $a_v=0$, for all
$v\in M(B)$. Hence $\{P^w:w\in W\}$ is a $\mathbb{Z}$-basis for $M_0(W,S)$ and then $V_W \simeq M_0(W,S)$ as $\mathbb{Z}$-modules.
The result follows since $W=W^M$, $\{P^s:s\in S\}$ are idempotents which generate $M_0(W,S)$ and, by Lemma \ref{lemma st}, they satisfy the same relations as the generators of the algebra  $\mathbb{Z}[W^M]$.
\end{proof}

By Theorem \ref{isomorfismi di algebre}, for any Coxeter system $(W,S)$ we have obtained a faithful representation of the monoid algebra $\mathbb{Z}[W^M]$, of dimension
$|W|$.

Since  $P^Jw \leqslant w$ for all $w\in W$, we can define a function $\mathcal{P}^J: \mathrm{Int}(W)\rightarrow \mathbb{Z}$ by $$\mathcal{P}^J(u,v):=\left\{
                                           \begin{array}{ll}
                                             1, & \hbox{if $u=P^Jv$;} \\
                                             0, & \hbox{otherwise,}
                                           \end{array}
                                         \right.$$ for all $[u,v]\in \mathrm{Int}(W)$. This implies the following corollary.

\begin{cor} \label{cor inject 1}
  Let $(W,S)$ be a Coxeter system. Then the assignment $P^J \mapsto \mathcal{P}^J$ gives an injective algebra morphism from $\mathbb{Z}[W^M]$ to the incidence algebra
  $I(W;\mathbb{Z})$.
\end{cor}

\begin{ex}
  Let $S_3$ be the symmetric group of order $6$ with generators $\{s,t\}$.
Then $M_0(S_3,\{s,t\})$ is the $\mathbb{Z}$-algebra generated by the identity
and the matrices $$P^{\{s\}}=\left(
                           \begin{array}{cccccc}
                             1 & 1 & 0 & 0 & 0 & 0 \\
                             0 & 0 & 0 & 0 & 0 & 0 \\
                             0 & 0 & 1 & 0 & 1 & 0 \\
                             0 & 0 & 0 & 1 & 0 & 1 \\
                             0 & 0 & 0 & 0 & 0 & 0 \\
                             0 & 0 & 0 & 0 & 0 & 0 \\
                           \end{array}
                         \right),~ P^{\{t\}}=\left(
                           \begin{array}{cccccc}
                             1 & 0 & 1 & 0 & 0 & 0 \\
                             0 & 1 & 0 & 1 & 0 & 0 \\
                             0 & 0 & 0 & 0 & 0 & 0 \\
                             0 & 0 & 0 & 0 & 0 & 0 \\
                             0 & 0 & 0 & 0 & 1 & 1 \\
                             0 & 0 & 0 & 0 & 0 & 0 \\
                           \end{array}
                         \right),
 $$ having chosen the following linear extension of the Bruhat order of $(S_3,\{s,t\})$:
$e<s<t<st<ts<sts$.
\end{ex}

\begin{oss}
 By the factorizations $w=w^Jw_J$ and its left version, one can see  that $P^Jw\leqslant_R w$ and $Q^Jw\leqslant_L w$, where $\leqslant_R$ and $\leqslant_L$ are the right
 weak order and the left weak order respectively (see \cite[Chapter~3]{BB}). Therefore the representation of the Coxeter monoid algebra $\mathbb{Z}[W^M]$ given by the
 endomorphisms $P^u$ realizes an injection into the incidence algebra $I(W,\leqslant_R;\mathbb{Z})$ and the one given by the endomorphisms $Q^u$ realizes an injection into
 the incidence algebra $I(W,\leqslant_L;\mathbb{Z})$. These algebras are isomorphic and they are subalgebras of $I(W;\mathbb{Z})$, since $(W,\leqslant_L) \simeq
 (W,\leqslant_R)$ and both are subposets of $(W,\leqslant)$.
\end{oss}

In the next lemma we prove that the idempotents in the Coxeter monoid $W^M$ are the maxima of the finite parabolic subgroups (see also \cite[Example~3.9]{denton1} and
\cite[Theorem~9]{KenneyCoxetermonoid}).
\begin{lem}
 The endomorphism $P^u$ is idempotent if and only if $u=w_0(J)$, for some $J\subseteq S$ such that $|W_J|<\infty$.
\end{lem}
\begin{proof}
  We have already proved that $P^{w_0(J)}=P^J$.
  So let $P^uP^u=P^u$. If $u\neq w_0(J)$ for all $J\subseteq S$ such that $|W_J|<\infty$, then
  there exists $s\in S$ such that $su>u$ and $s<u$.
  Therefore by Proposition \ref{lemmaPu} $P^uP^u(su)=P^us=e\neq s= P^u(su)$.
\end{proof}

%By the previous lemma one can see that the free $\mathbb{Z}$-module $M_0^v(W,S)$ with basis $\{P^u:u\in [e,v]\}$ is a subalgebra of $M_0(W,S)$
%  if and only if $v=w_0(J)$ for some $J\subseteq S$. In fact if $v=w_0(J)$ then $M_0^{w_0(J)}(W,S)= M_0(W_J,J)$.
%  On the other direction, if $M_0^v(W,S)$ is a subalgebra then $s\leqslant v$ and $vs>v$ imply $P^vP^s=P^{vs} \in M_0^v(W,S)$, a contradiction, since $vs\not \in [e,v]$.
%  Therefore $s \leqslant v$ implies $s\in D_R(v)$, i.e. $v=w_0(J)$ for some $J \subseteq S$.

We can define a class of $\mathbb{Z}[W^M]$-submodules of $V_W$ in the following manner. For any $v\in W$, by the regressivity of $P^J$, the
$\mathbb{Z}$-submodule of $V_W$
\begin{equation}\label{Vv}
  V_v:=\spn_{\mathbb{Z}}[e,v]
\end{equation} is a $\mathbb{Z}[W^M]$-submodule of $V_W$; moreover, by Lemma \ref{proiezionicommutano} and the left version of Proposition \ref{ordinepreservato}, the
$\mathbb{Z}$-module endomorphisms $Q^J \in \End(V_W)$ are $\mathbb{Z}[W^M]$-module endomorphisms of $V_v$, i.e. $Q^J \in \End_{\mathbb{Z}[W^M]}(V_v)$.
So the image of $Q^J$ is a $\mathbb{Z}[W^M]$-submodule of $V_W$; we define this image by
$$V^{J,v}:=\spn_{\mathbb{Z}} \{z \in {^J}W: z\leqslant Q^Jv\}.$$
Thus, for any $J\subseteq S$, the $\mathbb{Z}[W^M]$-modules $V_v$ decompose as $$V_v = V^{J,v} \oplus \spn_{\mathbb{Z}}\{u-Q^Ju:u\in {^{\setminus J}}[e,v]\},$$ where we have
defined ${^{\setminus J}}[u,v]:=\{z \in W\setminus {^J}W:u\leqslant z\leqslant v\}$, for all $u,v\in W\setminus {^J}W$.
By Lemma \ref{proiezionicommutano} we can also define, for any $J \subseteq S$, subalgebras of $\End(V_W)$ by
\begin{eqnarray*}
  &&M^J(W,S) := \{Q^Ja:a\in M(W,S)\}, \\
  &&M^{\setminus J}(W,S) := \{a-Q^Ja:a\in M(W,S)\}.
\end{eqnarray*} Then we have the following isomorphism of algebras: $$M(W,S) \simeq M^J(W,S) \oplus M^{\setminus J}(W,S).$$

Given an idempotent $P\neq \I$, let us define the idempotent $\overline{P}:=\I-P$ and let $\overline{\I}=\I$. For any $v\in W$ with reduced expression $s_1s_2 \cdots s_k$,
the endomorphism $(\I-P^{s_1})$ $(\I-P^{s_2})$ $\cdots (\I-P^{s_k})$ will be denoted by $\overline{P}^v$. The following proposition shows that $\overline{P}^v$ is well
defined, since it is independent from the choice of the reduced expression of $v$.
\begin{prop} \label{Pmoebius} Let $(W,S)$ be a Coxeter system. Then
  $$\overline{P}^v=\sum \limits_{u\leqslant v}(-1)^{\ell(u)}P^u$$ and
  $$P^v=\sum \limits_{u\leqslant v}(-1)^{\ell(u)}\overline{P}^u,$$ for all $v\in W$.
\end{prop}
\begin{proof}
  We proceed by induction on $\ell(v)$. If $v=e$ the result is obvious.
  Let $\ell(v)>1$ and $s_1s_2 \cdots s$ be a reduced expression for $v$. Then, by the inductive hypothesis,
  \begin{eqnarray*}
    \overline{P}^{s_1}\overline{P}^{s_2}\cdots \overline{P}^s &=& \overline{P}^{vs}(\I-P^s) \\
    &=& \sum \limits_{u\leqslant vs}(-1)^{\ell(u)}P^u-\left(\sum \limits_{u\leqslant vs}(-1)^{\ell(u)}P^u\right)P^s \\
    &=& \sum \limits_{u\leqslant vs}(-1)^{\ell(u)}P^u-\sum \limits_{\substack{u\leqslant vs\\ us<u}}(-1)^{\ell(u)}P^u-\sum \limits_{\substack{u\leqslant vs\\
    u<us}}(-1)^{\ell(u)}P^{us} \\
    &=& \sum \limits_{\substack{u\leqslant vs\\u<us}}(-1)^{\ell(u)}P^u+\sum \limits_{\substack{u\leqslant v\\ us<u}}(-1)^{\ell(u)}P^u \\ &=&
     \sum \limits_{\substack{u\leqslant v\\u<us}}(-1)^{\ell(u)}P^u+\sum \limits_{\substack{u\leqslant v\\ us<u}}(-1)^{\ell(u)}P^u \\&=&
    \sum \limits_{u\leqslant v}(-1)^{\ell(u)}P^u,
  \end{eqnarray*} since, by Proposition \ref{sollevamento}, $\{u\in [e,vs]:u<us\}=\{u\in [e,v]:u<us\}$. The second assertion can be proved by the same argument.
\end{proof}

\begin{oss}
  The involution $P^s\mapsto \overline{P}^s$ defines an involution on the $0$-Hecke algebra analogous to the involution $\iota$ on $\mathcal{H}(W,S)$, which is not defined
  for $q=0$. Compare the expression of $\iota$ in the standard basis of the Hecke algebra and the $R$-polynomials at $q=0$ (see, e.g., \cite[Sections~5.1 and~6.1]{BB}) with
  the result of Proposition \ref{Pmoebius}.
\end{oss}

\section{Representations of a right-angled Artin group in an incidence algebra} \label{4}

In this section we prove that the Hecke algebra of a right-angled Coxeter group injects in the $\mathbb{Z}[q,q^{-1}]$-algebra of the corresponding Coxeter monoid and that the
function $s \mapsto T_s$ provides an embedding of a right-angled Artin group $R^A$
into the Hecke algebra of the Coxeter system $(R,S)$ and then, by Corollary \ref{cor inject 1},
in the incidence $\mathbb{Z}[q,q^{-1}]$-algebra of the Bruhat poset of $R$. We recall that we have defined $A:=\mathbb{Z}[q,q^{-1}]$.

Given a right-angled Coxeter system $(R,S)$ the functions $f^q: S \rightarrow \End(A\otimes_{\mathbb{Z}} V_R)$
and $f^{-1}:S \rightarrow \End(A\otimes_{\mathbb{Z}} V_R)$ defined by
$$f^q(s)=q\I-(q+1)P^s$$ and
$$f^{-1}(s)=-\I+(q+1)P^s$$ for all $s\in S$
give two representations $\sigma^{q,R} : \mathcal{H}(R)\rightarrow \End(A\otimes_{\mathbb{Z}} V_R)$ and
 $\sigma^{-1,R} : \mathcal{H}(R)\rightarrow \End(A\otimes_{\mathbb{Z}} V_R)$ respectively.
Note that $f^{-1}(s)=q\I-(q+1)\overline{P}^s$ and $f^q(s)=-\I+(q+1)\overline{P}^s$, so that $\overline{f^{-1}(s)}=f^q(s)$, for all $s\in S$.
These functions, when defined on the set of generators of Coxeter systems of other types, do not provide
representations of their Hecke algebras. In fact we have the following results.
\begin{prop} \label{prop hecke}
  Let $(W,S)$ be a Coxeter system. Then
  $$[f^q(s)]^2 = q\I+(q-1)f^q(s),$$
  $$[f^{-1}(s)]^2 = q\I+(q-1)f^{-1}(s)$$ and
  $$ [f^q(s)f^q(t)]^nv=(-q)^n[v-n(q^{-1}+1)vs]+ke,$$
  $$ [f^{-1}(s)f^{-1}(t)]^nv=(-q)^n[v-n(q+1)vs]+k'e,$$ for all $n>0$, $s,t \in S$ such that $m(s,t)>2$,
  where $v:=ts$ and $k,k'\in  A$.
\end{prop}
\begin{proof}
  The first two equalities follow by a direct computation.
  %To prove the second ones let $v:=ts$. We claim that
  %$$((q\I-(q+1)P^s)(q\I-(q+1)P^t))^mv=(-q)^{m-1}(-qv+m(q+1)vs)+ke$$ and
  %$$((-\I+(q+1)P^s)(-\I+(q+1)P^t))^mv=(-q)^m(v-m(q+1)vs)+k'e,$$ for all $m\in \mathbb{P}$ and some $k,k'\in  \mathbb{Z}[q^{-1},q]$.
  We prove the second ones by induction on $n$. Let us prove the first equality.
  If $n=1$ then $(q\I-(q+1)P^s)(q\I-(q+1)P^t)v=-qv+(q+1)t$ and the result is true. Let $n>1$. Therefore
  \begin{eqnarray*}
   && ((q\I-(q+1)P^s)(q\I-(q+1)P^t))^nv  \\ &=& (q\I-(q+1)P^s)(q\I-(q+1)P^t)((-q)^{n-2}(-qv+(n-1)(q+1)t)+ke) \\
    &=& (-q)^{n-2}(q\I-(q+1)P^s)(qv+(n-1)q(q+1)t))+k'e \\
    &=& (-q)^{n-2}(q^2v-nq(q+1)t)+k''e  \\ &=&(-q)^{n-1}(-qv+n(q+1)t)+k''e,
  \end{eqnarray*} for some $k,k',k'' \in A$.

We prove now the second equality. If $n=1$ then $(-\I+(q+1)P^s)(-\I+(q+1)P^t)v=-q(v-(q+1)t)$ and the result is true. Let $n>1$. Therefore, by the inductive hypothesis,
  \begin{eqnarray*}
   && ((-\I+(q+1)P^s)(-\I+(q+1)P^t))^nv  \\ &=& (-\I+(q+1)P^s)(-\I+(q+1)P^t)((-q)^{n-1}(v-(n-1)(q+1)t)+ke) \\
    &=& (-q)^{n-1}(-\I+(q+1)P^s)(qv+(n-1)(q+1)t))+k'e \\
    &=& (-q)^{n-1}(-qv+nq(q+1)t)+k''e  \\ &=&(-q)^n(v-n(q+1)t)+k''e,
  \end{eqnarray*} for some $k,k',k'' \in A$.
\end{proof} When $m(s,t)>2$, by specialization at $q=1$, the results of Proposition \ref{prop hecke}  implies that $[f^x(s)f^x(t)]^n|_{q=1} \neq \I$, for all $x\in
\{-1,q\}$, $n>0$.
Therefore, in the Hecke algebra, $$\underbrace{f^x(s)f^x(t) \cdots f^x(s)}\limits_{\mbox{n times}} \neq \underbrace{f^x(t)f^x(s) \cdots f^x(t)} \limits_{\mbox{n times}},$$
for all $x\in \{-1,q\}$, $n >0$. This means that when $(W,S)$ is not right-angled then $f^q$ and $f^{-1}$ do not provide
representations of its Hecke algebras (not even of its group algebra, when $q=1$). Note that for $q=0$ they realize the representation of the Coxeter monoid algebra studied
in the previous section, in all types.

The representations of the Hecke algebra of a right-angled Coxeter group defined above are faithful;
this is proved in the next theorem.

\begin{thm} \label{rep faith}
  The representation $\sigma^{x,R}$ is faithful, for all $x\in \{-1,q\}$.
\end{thm}
\begin{proof}
  Let $a:=\sum \limits_{w\in B} a_wT_w \in \mathcal{H}(R)$, $B\subseteq R$, $|B|<\infty$ and $a_w\in A\setminus\{0\}$ for all $w\in B$.
  Let $M(B)$ be the set of maximal elements in $B$. Then $\sigma^{-1,R}(a)=\sum \limits_{v\in M(B)}(q+1)^{\ell(v)}a_vP^v+a'$
  for some $a' \in A\otimes_{\mathbb{Z}}M_0(R,S)$ which does not lie in the span of $M(B)$. By Theorem \ref{isomorfismi di algebre} and the flatness of $A$
  over $\mathbb{Z}$, we conclude that $\sigma^{-1,R}(a)=0$ implies $a_w=0$ for all $w\in B$. The case $x=q$ is analogous.
\end{proof} Theorem \ref{rep faith} gives an injection of the Hecke algebra of $R$ into the monoid algebra of $R^M$ over the ring $A$. Therefore, extending to the ring $A$
the result of Corollary \ref{cor inject 1}, we obtain the next corollary.
\begin{cor} \label{cor inject 2}  Let $R$ be a right-angled Coxeter group. Then we have the following injections of $A$-algebras:
  $$ \mathcal{H}(R) \hookrightarrow A[R^M] \hookrightarrow I(R;A).$$
\end{cor}

Now we consider a right-angled Artin group $R^A$. We want to define an infinite dimensional faithful representation
$\Sigma^{x,t}: R^A \rightarrow \Aut(\mathbb{Q}\otimes_{\mathbb{Z}}V_R)$, for any $t\in \mathbb{Q}\setminus \{-1,0,1\}$, $x\in \{-1,q\}$. We need the following
proposition, whose statement can be easily verified.
\begin{prop} \label{propPotenze}Let $V$ be a $A$-module and $P \in \End(V)$ an idempotent. Then
\begin{eqnarray*}
  (q\I-(q+1)P)^n &=& q^n\I-(q^n-(-1)^n)P, \\
  (-\I+(q+1)P)^n &=& (-1)^n\I+(q^n-(-1)^n)P,
\end{eqnarray*} for all $n\in \mathbb{Z}$.
\end{prop}

Let $\mathcal{H}^*(R)$ be the group of invertible elements of the Hecke algebra of $R$. The next theorem asserts
that the group morphism sending $s \in S$ to $T_s \in \mathcal{H}^*(R)$ for all $s\in S$ provides an injective group morphism from $R^A$ to $\mathcal{H}^*(R)$.

\begin{thm} \label{rappresentazioni Artin}
  The group morphism $\phi : R^A \rightarrow \mathcal{H}^*(R)$ defined on the generators by $\phi(s)=T_s$ for all $s\in S$, is injective.
Moreover, specializing at $q=t\in \mathbb{Q}\setminus \{-1,0,1\}$, it gives a faithful representation $\Sigma^{x,t}:R^A \rightarrow \Aut(\mathbb{Q} \otimes_{\mathbb{Z}}
V_R)$,
for all  $x\in \{-1,q\}$.
\end{thm}
\begin{proof} We recall that a reduced word in a right-angled Artin group is a word with minimum number of syllables (see Section \ref{1}). Let $s^{h_1}_1s^{h_2}_2 \cdots s^{h_k}_k$ be any reduced word for $w \in R^A$, where $w$ is not the identity.
Putting $h_i=1$ for all $i\in [k]$, by the minimality of $k$ we obtain a reduced word $s_1s_2 \cdots s_k$ in the Coxeter group $R$. By Theorem \ref{teorema artin} this operation
defines a function $R^A \rightarrow R$.
Notice that, by Proposition \ref{propPotenze}, we have that
$$f^q(s_i^{h_i})=[f^q(s_i)]^{h_i}=q^{h_i}\I-(q^{h_i}-(-1)^{h_i})P^{s_i}$$ and
$$f^{-1}(s_i^{h_i})=[f^{-1}(s_i)]^{h_i}=(-1)^{h_i}\I-(q^{h_i}-(-1)^{h_i})P^{s_i}.$$ Hence we obtain
\begin{eqnarray*}
  \sigma^{q,R}(\phi(w)) &=& (-1)^k(q^{h_1}-(-1)^{h_1})\cdots (q^{h_k}-(-1)^{h_k}) P^{s_1\cdots s_k}+a, \\
  \sigma^{-1,R}(\phi(w)) &=& (q^{h_1}-(-1)^{h_1})\cdots (q^{h_k}-(-1)^{h_k}) P^{s_1 \cdots s_k}+a',
\end{eqnarray*} for
some $a,a' \in M_0(R,S)$ independent of $P^{s_1 \cdots s_k}$. Both right hand sides of the above expressions are not the identity by Theorem \ref{isomorfismi di algebre}. Thus
we conclude that $ \sigma^{q,R}(\phi(s^{h_1}_1s^{h_2}_2 \cdots s^{h_k}_k))\neq \I$ and $ \sigma^{-1,R}(\phi(s^{h_1}_1s^{h_2}_2 \cdots s^{h_k}_k))\neq \I$.
Therefore, by Theorem \ref{rep faith}, we have that $\phi(s^{h_1}_1s^{h_2}_2 \cdots s^{h_k}_k)\neq \I$ and then $\phi$ is injective.

 We have proved that $\sigma^{x,R}\circ \phi : R^A \rightarrow \Aut(A\otimes_{\mathbb{Z}} V_R)$ is injective, for all $x\in \{-1,q\}$; since
 $t \in \mathbb{Q}\setminus \{-1,0,1\}$ implies $(t^{h_1}-(-1)^{h_1})\cdots (t^{h_k}-(-1)^{h_k})\neq 0$, by specialization, we obtain the
faithful representations stated.
\end{proof}
%\begin{oss}
%  The injection of Theorem \ref{rappresentazioni Artin} can also be deduced directly from the formula (in any Hecke algebra $\mathcal{H}(W,S)$)
%  $$(-T_s)^n=q[n-1]_{-q}T_e-[n]_{-q}T_s,$$ where $s\in S$, $[n]_q:=1+q+...+q^{n-1},$ for all $n>0$, and $[0]_q:=0$.
%\end{oss}
Let us denote with $I^*(R;A)$ the group of invertible elements of the incidence algebra $I(R;A)$ and with $A[R^M]^*$ the one of invertible elements of the
monoid $A$-algebra of $R^M$. From Theorem \ref{rappresentazioni Artin} and Corollary \ref{cor inject 2} we can deduce the next result.
\begin{cor} \label{corollario iniezioni} Let $R$ be a right-angled Coxeter group. Then we have the following injections of groups:
  $$R^A \hookrightarrow \mathcal{H}^*(R) \hookrightarrow A[R^M]^* \hookrightarrow I^*(R;A).$$
\end{cor}

We end noting that the modules $\mathbb{Q} \otimes_{\mathbb{Z}} V_v$, where $V_v$ are the  $\mathbb{Z}[W^M]$-modules defined in Equation \eqref{Vv}, give rational
representations of the Artin group $R^A$, of dimension $|[e,v]|< \infty$, for any $v\in R$. In fact $A\otimes_{\mathbb{Z}} V_v$ are modules of the Coxeter monoid algebra
$A[R^M]$, and then of the Hecke algebra of $R$, by Theorem \ref{rep faith}, and so the group $R^A$ acts on $\mathbb{Q} \otimes_{\mathbb{Z}} V_v$, by Theorem
\ref{rappresentazioni Artin}. For example, when $v=e$ we have that $\Sigma^{q,t}$ gives the alternating representation, for all $t\in \mathbb{Q}\setminus \{-1,0,1\}$.

\section*{Acknowledgements}
I am grateful to Francesco Brenti and Nicol\'as Libedinsky for their useful suggestions. In particular, the existence and the importance of the main problem investigated here
was pointed out to me by Prof. Libedinsky. I also want to thank Matthieu Calvez for some remarks.

This research was supported by Postdoctorado FONDECYT-CONICYT number 3160010.

\end{document}